\documentclass[12pt,a4paper]{amsart}
\usepackage{graphicx,amssymb}
\usepackage{amsmath,amssymb,amscd}
\input xy
\xyoption{all}
\usepackage[all]{xy}
\usepackage{hyperref}

\newcommand{\hra}{\hookrightarrow}
\newcommand{\ra}{\rightarrow}

\newcommand{\CC}{\mathbb C}
\newcommand{\ZZ}{\mathbb Z}
\newcommand{\QQ}{\mathbb Q}

\newcommand{\PP}{\mathbb P}
\newcommand{\FF}{\mathbb F}

\newcommand{\cE}{\mathcal{E}}
\newcommand{\cF}{\mathcal{F}}

\newcommand{\cO}{\mathcal{O}}

\newcommand{\cK}{\mathcal{K}}

\newcommand{\cH}{\mathcal{H}}
\newcommand{\image}{\mbox{Im}}
\newcommand{\Ker}{\mbox{Ker}}
\newcommand{\Pic}{\mbox{Pic}}
\newcommand{\Ext}{\mbox{Ext}}

\newcommand{\Hom}{\mbox{Hom}}

\newcommand{\Coker}{\mbox{Coker}}

\newcommand{\ch}{\operatorname{ch}}
\theoremstyle{plain}
\newtheorem{theorem}{Theorem}[section]
\newtheorem{lem}[theorem]{Lemma}
\newtheorem{prop}[theorem]{Proposition}
\newtheorem{cor}[theorem]{Corollary}
\newtheorem{conj}[theorem]{Conjecture}
\newtheorem{rem}[theorem]{Remark}

\numberwithin{equation}{section}

\begin{document}
\title[Brill-Noether loci in rank 2]{Non-emptiness of Brill-Noether loci in $M(2,K)$}

\author{Herbert Lange}
\author{Peter E. Newstead}
\author{Seong Suk Park}

\address{H. Lange\\Department Mathematik\\
              Universit\"at Erlangen-N\"urnberg\\
              Cauerstrasse 11\\
              D-$91058$ Erlangen\\
              Germany}
              \email{lange@mi.uni-erlangen.de}
\address{P.E. Newstead\\Department of Mathematical Sciences\\
              University of Liverpool\\
              Peach Street, Liverpool L69 7ZL, UK}
\email{newstead@liv.ac.uk}
\address{Seong Suk Park\\Facultad de Formacion del Profesorado y Educacion\\
C/Francisco Tomas y Valiente, 3\\ Universidad Autonoma de Madrid\\
28049 Madrid, Spain}
\email{seong.park@uam.es}
\date{\today}

\thanks{The first two authors are members of the research group VBAC (Vector Bundles on Algebraic Curves). The second author 
would like to thank the Department Mathematik der Universit\"at 
         Erlangen-N\"urnberg for its hospitality}

\begin{abstract}
Let $C$ be a smooth projective complex curve of genus $g \geq 2$. We investigate the Brill-Noether locus consisting of stable bundles of rank 2 and canonical determinant
having at least $k$ independent sections. Using the Hecke correpondence we construct a fundamental class, which determines the non-emptiness of this locus at least 
when $C$ is a Petri curve. We prove that in many expected cases the Brill-Noether locus is non-empty. For some values of $k$ the result is best possible.
\end{abstract}
\maketitle

\section{Introduction}\label{intro}

Let $C$ be a smooth projective complex curve of genus $g \geq 2$ with canonical bundle $K$. Let $M(2,K)$ be the moduli space of stable bundles of rank 2 with determinant $K$
and $k$ a positive integer. 
The Brill-Noether locus $B(2,K,k)$ in $M(2,K)$ is defined by 
$$B(2,K,k) := \{ E \in M(2,K) \;|\;h^0(E) \geq k \}.
$$
This is a degeneracy locus of expected dimension 
$$
\beta(2,K,k) := 3g-3 - \frac{k(k+1)}{2}
$$
(see \cite{m} and \cite{bf}). Bertram and Feinberg made the following conjecture which Mukai also stated as a problem (see \cite[Problem 4.11]{m} and \cite[Problem 4.8]{m1}).\\

\noindent
{\bf Conjecture.} (i) {\it If $\beta(2,K,k) \geq 0$, then $B(2,K,k) \neq \emptyset$.}\\
(ii) {\it For $C$ a general curve, if $B(2,K,k) \neq \emptyset$, then $\beta(2,K,k) \geq 0$}.\\

Part (ii) of the conjecture was proved by Bertram-Feinberg and Mukai in some low genus cases. Recently a complete proof of (ii) was given 
by Teixidor \cite{t2} by showing that the canonical Petri map, 
which governs the infinitesimal behaviour of $B(2,K,k)$, is injective. It may be noted that the generality hypothesis is necessary. In fact, there exist Petri curves for which (ii) fails 
(see \cite{v}).

The subject of this paper is part (i) of the conjecture. This is known for a general curve when $g \leq 12$ \cite{bf}. More recently, Teixidor \cite{t1} proved the conjecture for a general 
curve under the assumption that  $g\ge\frac{k^2}4$.
Her proof proceeds by deformation from a reducible nodal curve. We adopt a different approach which is hinted at in \cite{bf} and \cite{m1}. The idea is to define a fundamental class
for $B(2,K,k)$ and prove that it is non-zero. The problem with this method is that $M(2,K)$ is only quasiprojective and its natural compactification $\overline M(2,K)$ is singular.
To overcome this, we consider the Hecke correspondence $H$ (see Section 2 for details) which is a smooth projective variety. One can define a Brill-Noether locus $B_H(k)$ in $H$, 
which has a fundamental class $b_H(k)$. Our first main result is\\

\noindent
{\bf Theorem \ref{thm2.5}.} {\it Suppose that $C$ is a general curve of genus $g\ge3$ and $k$ is a positive integer, $(g,k)\ne(4,4)$. Then}
$$
B(2,K,k) \neq \emptyset \iff b_H(k) \neq 0.
$$

\vspace{0.5cm}
In fact, the implication $b_H(k) \neq 0 \Rightarrow B(2,K,k) \neq \emptyset$ is valid on any Petri curve for any $g \geq 2$ except for the cases $g=k=2$ and $g=k=4$.
As a consequence of Theorem \ref{thm2.5} and of the work of \cite{bf} and \cite{m} we obtain a complete proof of part (i) of the conjecture for a general curve $C$ of 
genus $g \geq 3$ when $ k \leq 7$ (Corollary \ref{cor2.10}).

For Petri curves the problem now reduces to showing that $b_H(k) \neq 0$ whenever $\beta(2,K,k)\ge0$. Since  the cohomology of $H$ is completely known, this is a purely computational problem.
The computations are however difficult and we are not able to complete them in all cases. Following some preliminary work in Section 4, we consider the 
case when $g$ is a prime in Section 5. The computations simplify in this case and we can prove\\

\noindent
{\bf Theorem \ref{thm4.3}.} 
{\it Suppose $g$ is an odd prime. If $g-1 \geq \max \left\{ \frac{k(k-1)}{4}, 2k-1 \right\}$, then $b_H(k) \neq 0$}.\\

As a consequence we prove that for $k \geq 8$ and $g \geq g_k$, where $g_k$ is the smallest prime satisfying $g_k - 1 \geq \frac{k(k-1)}{4}$, $B(2,K,k) \neq \emptyset$ 
on any Petri curve $C$ (Corollary \ref{cor4.5}).

The condition $g-1 \geq \frac{k(k-1)}{4}$ is better than Teixidor's condition, but we require a prime number $g_k$ in the statement of Corollary \ref{cor4.5}. However,
in some cases our results improve on those of Teixidor. Moreover, our results apply to an arbitrary Petri curve, while those of Teixidor are valid only for a general 
curve in a somewhat unspecified sense.

When $g-1 < \frac{k(k-1)}{4}$, Theorem \ref{thm4.3} is not applicable, but the earlier part of Section 5 is still valid. The computations are more complicated and 
we have completed them only by using Maple for $10 \leq k \leq 24$ (see Theorem \ref{thm6.1} and the subsequent remarks). These results are considerably stronger than those of Teixidor and in some cases are best possible.

In Section 7, we consider the moduli space $M(2, K(p))$ of stable vector bundles of rank 2 and determinant $K(p)$ and the Brill-Noether locus
$$B(2,K(p),k):=\{E\in M(2,K(p))|h^0(E)\ge k\}.$$
We prove\\

\noindent{\bf Theorem \ref{thm7.1}.}
{\em Let $C$ be a smooth projective complex curve of genus $g \geq 2$ and $p \in C$ a fixed point. For a positive integer $k$, if $b_H(k) \neq 0$, then $B(2,K(p),k) \neq \emptyset$
and every component $X$ has dimension 
$$
\dim(X) \geq \beta(2,K,k) + 1.
$$
If $C$ is general of genus $g \geq 3$, then, for $k \geq 2$,
$$
\dim B(2,K(p),k) \leq \beta(2,K,k) +k.
$$
}

Throughout the paper, $C$ is a smooth projective complex curve of genus $g\ge2$ with canonical bundle $K$. We write $B(1,d,k)$ for the Brill-Noether locus
$$B(1,d,k):=\{M\in\Pic^d(C)|h^0(M)\ge k\}$$
and
$$\beta(1,d,k):=g-k(k-d+g-1)$$
for the expected dimension of $B(1,d,k)$. We recall that $C$ is a {\em Petri} curve if the multiplication map
$$H^0(M)\otimes H^0(K\otimes M^{-1})\ra H^0(K)$$
is injective for every line bundle $M$ on $C$.

\section{Preliminaries}

Let $M(2,K)$ and $B(2,K,k)$ be as in the introduction. The variety $M(2,K)$ has a natural compactification
$\overline M(2,K)$ parametrizing S-equivalence classes of semistable vector bundles of rank 2 and determinant $K$. For $p$ a fixed point of $C$, the moduli space $M(2,K(p))$ is a smooth projective variety and supports a universal bundle $\cE$ on $C \times M(2,K(p))$
such that $\wedge^2 \cE = K(p) \boxtimes L$ with $L$ a line bundle on $M(2,K(p))$. In fact $\Pic(M(2,K(p)))\simeq\ZZ$ and we can choose $L$ to be the ample generator of $\Pic(M(2,K(p)))$.

Let $\cE_p$ denote the restriction of $\cE$ to $\{p\} \times M(2,K(p))$. The Hecke correspondence $H$ is defined by 
$$
H := \PP(\cE_p).
$$
Denote by $\pi_1: H \ra M(2,K(p))$ the natural projection.
The variety $H$ can be thought of as a moduli space of parabolic bundles and parametrizes exact sequences
\begin{equation} \label{eq2.1}
0 \ra F \ra E \ra \CC(p) \ra 0,
\end{equation}
where $E \in M(2,K(p))$ and $\CC(p)$ is the skyscraper sheaf at $p$. In fact, this is a fine moduli space and we have a universal object
$$
0 \ra \cF \ra (1 \times \pi_1)^*\cE \ra \cO_H(1) \ra 0
$$
on $C \times H$, where $\cO_H(1)$ is considered as a sheaf on $C \times H$ supported on $\{p\} \times H$. Note that in the sequence \eqref{eq2.1} the bundle $F$ is always
semistable. So we have a morphism $\pi_2: H \ra \overline M(2,K)$. Over $M(2,K)$ the fibres of $\pi_2$ are isomorphic to $\PP^1$. 

We define the Brill-Noether locus $B_H(k)$ set-theoretically by
$$
B_H(k) := \{ 0 \ra F \ra E \ra \CC(p) \ra 0  \; | \; h^0(F) \geq k \}.
$$
In fact, $B_H(k)$ has the structure of a degeneracy locus and is hence a subscheme of $H$. To see this, fix an effective divisor $D$ of degree $> g-1$ on $C$. We denote the pullback of 
this divisor to $C \times H$ also by $D$. Then we have an exact sequence
$$
0 \ra \cF(-D) \ra \cF(D) \ra \cF(D)/\cF(-D) \ra 0
$$
on $C \times H$. Taking direct images by the projection $p_2: C \times H \ra H$, we have an exact sequence
$$
0 \ra {p_2}_*\cF(D) \ra {p_2}_*(\cF(D)/\cF(-D)) \ra R^1p_2\cF(-D) \ra 0.
$$
Moreover, there is a natural inclusion 
$$
{p_2}_*(\cF/\cF(-D)) \hra {p_2}_*(\cF(D)/\cF(-D)).
$$
Then ${p_2}_*(\cF/\cF(-D))$ and ${p_2}_*(\cF(D))$ are Lagrangian subbundles of 
${p_2}_*(\cF(D)/\cF(-D))$ with respect to the skew-symmetric bihomomorphism
$$
\langle \cdot,\cdot \rangle: {p_2}_*(\cF(D)/\cF(-D)) \times {p_2}_*(\cF(D)/\cF(-D)) \ra \pi_1^*L,
$$
which is induced by $\cE(D) \times \cE(D) \ra \wedge^2 \cE(D) = K(p) \boxtimes L(2D)$. Note that for $x:= (0 \ra F \ra E \ra \CC(p) \ra 0) \in H$ we have
$$
{p_2}_*(\cF(D))_x \cap {p_2}_*(\cF/\cF(-D))_x \simeq H^0(F).
$$
Thus the degeneracy locus of Lagrangian subbundles gives the scheme structure of $B_H(k)$. The expected dimension of $B_H(k)$ is
$$
\beta(2,K,k) +1 = 3g-2-\frac{k(k+1)}{2}.
$$

We complete the section with a description of the cohomology of $H$. Following \cite{n} we can write the Chern classes of $\cE$ in the form
$$
c_1(\cE) = \alpha + (2g-1)\varphi
$$
$$
c_2(\cE) = \chi + \psi + g\alpha \otimes \varphi
$$
where $ \alpha = c_1(L), \; \chi \in H^4(M(2,K(p)),\ZZ), \; \psi \in H^3(M(2,K(p)),\ZZ) \otimes H^1(C,\ZZ)$ and $\varphi \in H^2(C,\ZZ)$ is the fundamental class. We write also
$$
\beta = \alpha^2 -4 \chi
$$
and define $\gamma \in H^6(M(2,K(p)),\ZZ)$ by
$$
\psi^2 = \gamma \otimes \varphi.
$$
We shall be concerned with the subalgebra $A_{M(2,K(p))}$ of $H^*(M(2,K(p)),\QQ)$ generated by $\alpha, \beta$ and $\gamma$, which can be written as 
$$
A_{M(2,K(p))} = \QQ[\alpha, \beta, \gamma]/I_g.
$$
Here the ideal of relations $I_g$ is explicitly described in \cite{kn}. For any polynomial $f \in \QQ[\alpha,\beta,\gamma]$ we denote by $(f)$ the corresponding cohomology class.

Let $h$ be the class of $\cO_H(1)$. Then the corresponding subalgebra $A_H$ of $H^*(H,\QQ)$ is generated by $\alpha, \beta,\gamma$ and $h$. Moreover, $A_H$ is a free module over 
$A_{M(2,K(p))}$ with basis $1,h$ and 
$$
h^2= \alpha h - \frac{\alpha^2 - \beta}{4}.
$$
(Note here that $\alpha$ and $\frac{\alpha^2 - \beta}{4}$ are the Chern classes of $\cE_p$.) For $f \in \QQ[h,\alpha,\beta,\gamma]$ we again denote by $(f)$ the corresponding 
cohomology class.

\begin{lem} \label{hr} 
For all integers $r \geq1$, 
 $$
h^r = \frac{1}{2^r} \left( \sum_{i \, {\small even} \, \leq r} \alpha^{r-i} \beta^{\frac{i}{2}} {r \choose i} + (2h-\alpha) 
\sum_{i \, {\small odd} \, \leq r} \alpha^{r-i} \beta^{\frac{i-1}{2}} {r \choose i} \right).
 $$
\end{lem}

\begin{proof}
This is true for $r = 1$ and $r=2$. The result follows by induction on $r$. 
\end{proof}

In general, it is not easy to see whether for a given $f$ the class $(f)$ is zero or not, since the relations among $\alpha, \beta, \gamma, h$ are complicated. However we have the 
following explicit formula when the expected dimension is zero.

\begin{prop} \label{thaddeus} (Thaddeus \cite{th}) When $m + 2n + 3p = 3g-3$, the intersection number is
$$
(\alpha^m \beta^n \gamma^p) = (-1)^{g-p} \frac{g!m!}{(g-p)!q!} 2^{2g-2-p}(2^q -2)B_q,
$$
where $q = m+p+1-g$ and $B_q$ is the $q$th Bernoulli number, defined by
$$
\frac{x}{e^x-1} = 1 - \frac{1}{2}x + \sum_{q \; even} B_q \frac{x^q}{q!}.
$$
We set $B_0 = 1$ and $B_q = 0$ for $q < 0$. 
\end{prop}

\section{Brill-Noether loci on the Hecke correspondence}

Our first object in this section is to obtain a formula for the (virtual) fundamental class $b_H(k)$ of $B_H(k)$. 
Set
$$
c_i := c_i(({p_2}_*\cF(D))^{\vee} \otimes \sqrt{L}) + c_i(({p_2}_*\cF|_D)^{\vee} \otimes \sqrt{L}),
$$
where $\sqrt{L}$ is a formal object with Chern class $c(\sqrt{L}) = 1 + \frac{\alpha}{2}$.
Then \cite[equation (4)]{f} says that the virtual fundamental class of $B_H(k)$ is
$$
b_H(k) = \Delta_{k,k-1,\dots,1}(c_i) = \left| \begin{array}{cccc}
                                               c_k & c_{k+1} & \dots & c_{2k-1}\\
                                               c_{k-2} & c_{k-1} & \dots & c_{2k-3} \\
                                               \dots &&& \\
                                               c_{-k+2} & c_{-k+3} & \dots & c_1
                                              \end{array} \right|
                                              $$
From the sequence defining $\cF$ we get that the Chern classes of $\cF$ are given by
$$
c_1(\cF) = \alpha + (2g-2)\varphi \in H^2(C \times H),
$$
$$
c_2(\cF) = \chi + \psi + (h + (g-1)\alpha) \otimes \varphi \in H^4(C \times H).
$$
By standard formulae for tensor product and Grothendieck-Riemann-Roch we have
\begin{equation} \label{e2.1}
 \ch_{2n-1}(({p_2}_*\cF(D))^{\vee} \otimes \sqrt{L}) = \frac{1}{(2n-1)!} \left( \frac{\beta h}{4} - \frac{n-1}{2} \gamma \right) \left( \frac{\beta}{4} \right)^{n-2}_,
\end{equation}
\begin{equation} \label{e2.2}
 \ch_{2n}(({p_2}_*\cF(D))^{\vee} \otimes \sqrt{L}) = \frac{2}{(2n)!} \left( \frac{\beta}{4} \right)^{n} \deg D.
\end{equation}
On the other hand, taking $D=q_1+\ldots,q_{\deg D}$ with all $q_i$ distinct,
\begin{eqnarray*}
 c(({p_2}_*\cF|_D)^{\vee} \otimes \sqrt{L}) &=& c(\oplus_{i=1}^{\deg D} ({p_2}_*\cF|_{q_i})^{\vee} \otimes \sqrt{L}))\\
& = & \prod_{i=1}^{\deg D} c(({p_2}_*\cF|_{q_i})^{\vee} \otimes \sqrt{L})\\
& = & \sum_{n=0}^{\deg D} {\deg D \choose n} \left( - \frac{\beta}{4} \right)^n.
\end{eqnarray*}
Hence 
\begin{equation}  \label{e2.3} 
  c_{2n-1}(({p_2}_*\cF|_D)^{\vee} \otimes \sqrt{L}) = 0,  \quad  c_{2n}(({p_2}_*\cF|_D)^{\vee} \otimes \sqrt{L}) = {\deg D \choose n} \left( - \frac{\beta}{4} \right)^n.
\end{equation}

Note that the cohomology class $b_H(k)$ can be written as a polynomial in $h,\beta,\gamma$ and $\deg D$. As a cohomology class it is independent of $\deg D > g-1$. It follows that $b_H(k)$ 
can be computed by putting $\deg D = 0$. So formulas \eqref{e2.1} to \eqref{e2.3} can be replaced by
\vspace{0.2cm}
\newline
$\eqref{e2.1}' \hspace{2.6cm} \ch'_{2n-1} = \frac{1}{(2n-1)!} \left( \frac{\beta h}{4} - \frac{n-1}{2} \gamma \right) \left( \frac{\beta}{4} \right)^{n-2}_,$
\vspace{0.2cm}
\newline
$\eqref{e2.2}' \hspace{5cm}
\ch'_{2n} = 0,
$
\vspace{0.2cm}
\newline
$\eqref{e2.3}' \hspace{3.5cm}
 c'_{0} = 1, \quad c'_i = 0 \quad \mbox{for} \quad i>0.$
\vspace{0.2cm}
\newline
We can now write $c_0 = 2$ and use $\eqref{e2.1}'$ and $\eqref{e2.2}'$ to define polynomials $c_i= c_i(h,\beta,\gamma)$ for $i \geq 1$ by means of the universal formulae relating Chern characters and Chern classes. We can then
define a polynomial $P_k \in \QQ[h,\beta, \gamma]$ by
\begin{equation} \label{e2.4} 
 P_k(h,\beta,\gamma) = \Delta_{k,k-1,\dots,1}(c_i).
\end{equation}
Moreover, using Lemma \ref{hr} we can express this as
$$
f(\alpha,\beta,\gamma) h + f'(\alpha, \beta, \gamma)
$$
with $f(\alpha,\beta,\gamma),f'(\alpha,\beta,\gamma) \in \QQ[\alpha,\beta,\gamma]$. It is important to note that these polynomials are independent of $g$.

To show that $b_H(k) \neq 0$ we need to prove that the cohomology classes $(f(\alpha,\beta,\gamma))$ and $(f'(\alpha, \beta, \gamma))$ are not both zero.
Since $\alpha, \beta, \gamma$ have degrees $2,4,6$ respectively, the degrees of $f$ and $f'$ are $k(k+1) -2$ and $k(k+1)$ respectively.

\begin{lem} \label{lnonzero}
Suppose that $\beta^i$ is the smallest power of $\beta$ with non-zero coefficient $N$ in $P_k(1,\beta,0)$. Then $i<\frac{k(k+1)}4$ and the coefficient of 
$\alpha^{\frac{k(k+1)}{2} -1 -2i}\beta^i$ in $f(\alpha,\beta,\gamma)$ is non-zero. 
\end{lem}

\begin{proof}
Write $\ell=\frac{k(k+1)}2-2i$. Note that, by Lemmas \ref{lem3.5} and \ref{lem3.6} in the next section, $P_k(1,\beta,0)$ is not identically zero and has degree $<\frac{k(k+1)}4$. Moreover $N$ is the coefficient of $h^\ell\beta^i$ in $P_k(h,\beta,\gamma)$. From Lemma \ref{hr}, the coefficient of $\alpha^{\ell -1}\beta^i$ 
in $f(\alpha,\beta,\gamma)$ is $\frac{\ell}{2^{\ell-1}}N$. 
\end{proof}

\begin{lem} \label{lem2.1}
Let $\cK$ be the locus of $S$-equivalence classes of strictly semistable bundles in ${\overline M}(2,K)$. Then $\pi_2^{-1}(\cK)$ has dimension $2g -1$.
\end{lem}
\begin{proof}
A typical element in $\pi_2^{-1}(\cK)$ has the form 
\begin{equation}\label{e2.6}
0 \ra F \ra E \ra \CC(p) \ra 0,
\end{equation}
where $F$ lies in an exact sequence
\begin{equation} \label{e2.5}
0 \ra M \ra F \ra K \otimes M^{-1} \ra 0
\end{equation}
with $M$ a line bundle of degree $g-1$. The non-trivial extensions \eqref{e2.5} are classified by $\PP(H^1(K^{-1} \otimes M^2))$
which has dimension $g-2$, unless $M^2 \simeq K$, in which case it has dimension $g-1$. Moreover the non-trivial extensions \eqref{e2.6} are classified by $\PP(\Ext^1(\CC(p),F))$, which has dimension $1$. The result follows.
\end{proof}

\begin{prop} \label{prop2.2}
Suppose that $\beta(2,K,k) = 3g-3 - \frac{k(k+1)}{2} \geq 2g-1$, i.e. $g \geq \frac{k(k+1)}{2} + 2$. Then $B(2,K,k) \neq \emptyset$. 
\end{prop}

\begin{proof}
We prove first that $b_H(k)$ is not zero.
We know from \cite{kn} that, since $k(k+1) < 2g$, there is no relation among the cohomology classes $\alpha, \beta, \gamma$ in 
$H^{k(k+1)}(M(2,K(p)))$. 
It follows at once from Lemma \ref{lnonzero} that the cohomology class represented by $f(\alpha,\beta,\gamma)$ is non-zero.
Hence $B_H(k)$ is a non-empty locus of dimension at least 
$$
\beta(2,K,k) + 1 \geq 2g.
$$
So $B_H(k) \not \subset \pi^{-1}_2(\cK)$ by Lemma \ref{lem2.1}, which gives the result.
\end{proof}

\begin{lem} \label{lem2.4}
Let $C$ be a Petri curve and let $\cK$ be the locus of $S$-equivalence classes of strictly semistable bundles in ${\overline M}(2,K)$. 
If either $k \geq 5$ or $g\ge5$ and $k\ge3$, then
$$
\dim \left( \pi_2^{-1}(\cK) \cap B_H(k) \right) < \beta(2,K,k) +1.
$$
\end{lem}

\begin{proof}
Suppose that $F$ is a strictly semistable bundle with $h^0(F) \geq k$. Then there exists a line bundle $M$ sitting in an exact sequence \eqref{e2.5}. 
Since $h^0(F) \geq k$ and $h^0(M) = h^0(K \otimes M^{-1})$, we get $h^0(M) \geq \lceil \frac{k}{2} \rceil$.

If $g < \lceil \frac{k}{2} \rceil^2$, then $\beta(1,g-1,\lceil \frac{k}{2} \rceil) < 0$. Since $C$ is Petri, this contradicts the existence of $M$.

So suppose $g \geq \lceil \frac{k}{2} \rceil^2$ and $M \in B(1,g-1,\lceil \frac{k}{2} \rceil)$. Suppose first that $k = 2m$. Then 
$$\dim B(1,g-1,m) = \beta(1,g-1,m) = g - m^2.$$ 
If $h^0(M) =m$, then all the sections of $K \otimes M^{-1}$ must lift to $F$. This implies that the class of \eqref{e2.5} belongs to 
$$
\Ker \left( H^1(K^{-1} \otimes M^2) \ra \Hom ( H^0(K \otimes M^{-1}) \ra H^1(M))\right).
$$
This is dual to the multiplication map 
$$
\mu: H^0(K \otimes M^{-1}) \otimes H^0(K \otimes M^{-1}) \ra H^0(K^2 \otimes M^{-2}).
$$
By the Hopf Lemma,
$$
\dim \image(\mu) \geq 2m-1.
$$
So
$$
\dim \Coker (\mu) \leq \left\{ \begin{array}{lcc}
                                g-2m & if & M^2 \not \simeq K,\\
                                g - 2m + 1 & if & M^2 \simeq K.
                               \end{array}  \right.
$$
In the second case $M$ is a theta characteristic which implies that $m \leq 1$, since $C$ is a Petri curve. So we can ignore this case. 
The line bundle $M$ belongs to $B(1,g-1,m) \setminus B(1,g-1,m+1)$. The corresponding exact sequences \eqref{e2.5} give a contribution to
$\pi_2^{-1}(\cK) \cap B_H(k)$ of dimension $\leq g-m^2 + g -2m = 2g - (m+1)^2 + 1.$
On the other hand, the exact sequences \eqref{e2.5} with line bundles in $B(1,g-1,m+1)$ give a contribution to $\pi_2^{-1}(\cK) \cap B_H(k)$ of dimension 
$\leq g - (m+1)^2 + g = 2g - (m+1)^2$. It follows that
\begin{eqnarray*}
\dim \left( \pi_2^{-1}(\cK) \cap B_H(k) \right) & \leq & 2g - (m+1)^2 + 1 \\
& < & 3g -2 - m(2m+1) = \beta(2,K,k) + 1,
\end{eqnarray*}
since $g \geq m^2 > m^2 -m + 2$ for $m \geq 3$ which holds by hypothesis when $k\ge5$. The inequality $g>m^2-m+2$ holds also for $m=2$ if $g\ge5$.

If $k = 2m+1$, then $M \in B(1,g-1,m+1)$. By assumption $g \geq (m+1)^2$. So $\dim B(1,g-1,m+1) = g - (m+1)^2$ and
\begin{eqnarray*}
\dim \pi_2^{-1}(\cK) &\leq& 2g - (m+1)^2 \\
&<& 3g-2 -(m+1)(2m+1) = \beta(2,K,k) + 1,
\end{eqnarray*}
since $g \geq (m+1)^2 > m(m+1) + 2$ for $m \geq 2$ which holds by hypothesis when $k\ge5$. The inequality $g>m(m+1)+2$ holds also for $m=1$ if $g\ge5$.
\end{proof}

\begin{theorem} \label{thm2.5}
Suppose that $C$ is a general curve of genus $g\ge3$ and $k$ is a positive integer, $(g,k)\ne(4,4)$. Then
$$
B(2,K,k) \neq \emptyset \iff b_H(k) \neq 0.
$$
\end{theorem}

\begin{proof}
Suppose first that either $k\ge5$ or $g\ge5$ and $k\ge3$. If $B(2,K,k) \neq \emptyset$, then it has the expected dimension $\beta(2,K,k)$ by \cite{t2}. It follows from this and Lemma \ref{lem2.4} that $B_H(k)$ has the expected dimension 
and therefore $b_H(k) \neq 0$.
Conversely, if $b_H(k) \neq 0$, then $B_H(k) \neq \emptyset$. By Lemma \ref{lem2.4}, $B_H(k) \not \subset \pi_2^{-1}(\cK)$. So $B(2,K,k) \neq \emptyset$.

If $k=4$ and $g=3$, then, by \cite{bf}, $B(2,K,k)=\emptyset$ and, since there are no line bundles $M$ of degree $2$ with $h^0(M)\ge2$, $B_H(k)$ does not meet $\pi_2^{-1}(\cK)$ and is therefore also empty. Moreover $b_H(k)=0$. If $k=3$ and $g=3$ or $4$, it is proved in \cite{bf} that $B(2,K,k)\ne\emptyset$. Moreover, for $g=4$, the earlier part of the proof still works to show that $b_H(k)\ne0$ since we need only a non-strict inequality in Lemma \ref{lem2.4} for this implication. For $g=3$, as before, $B_H(k)$ does not meet $\pi_2^{-1}(\cK)$ and is therefore of the expected dimension; so $b_H(k)\ne0$.

Finally, suppose $k=1$ or $2$. Then $b_H(k)\ne0$ by direct computation using Lemma \ref{hr} and the argument of Lemma \ref{lem2.4} works to show that $B(2,K,k)\ne\emptyset$ since $g\ge3$.
\end{proof}

\begin{rem}  \label{r2.6}
{\rm The implication $b_H(k)\ne0 \Rightarrow B(2,K,k) \neq \emptyset$ is valid on any Petri curve of genus $g\ge3$, except for the case $g=k=4$. In this case $\beta(2,K,4)=-1$, so $B(2,K,4)=\emptyset$ by \cite{bf}. On the other hand, a Petri curve of genus $4$ possesses two distinct trigonal bundles $T$ and $T'$ and there is a unique stable bundle $E$ fitting into an exact sequence
$$0\to T\oplus T'\to E\to {\mathbb C}(p)\to0.$$
Thus $B_H(4)$ consists of a single point and has the expected dimension; hence $b_H(4)\ne0$.  Note also that there exist Petri curves for which the opposite implication fails.
In fact, one can have $B(2,K,k) \neq \emptyset$ with $\beta(2,K,k)+1 < 0$ (see \cite{v}).}
\end{rem}

\begin{rem} \label{rem2.6}
{\rm For $g =2$, then certainly $B(2,K,k)=\emptyset$ and $b_H(k)=0$ for $k\ge3$. Direct calculations show that $b_H(k)\ne0$ for $k=1$ or $2$. However, by \cite{bgn}, $B(2,K,k) \neq \emptyset$ for $k = 1$, but is empty for $k = 2$.}
\end{rem}

Another consequence of Lemma \ref{lem2.4} concerns the Brill-Noether locus 
$$
\overline B(2,K,k) := \{ [F] \in \overline M(2,K) \;|\; h^0({\rm gr} (F)) \geq k \}
$$
in $\overline M(2,K)$. Here $[F]$ denotes the S-equivalence class of $F$ and ${\rm gr}(F)$ denotes the graded object associated to a Jordan-H\"older filtration of $F$.

\begin{prop} \label{BNbar}
Let $C$ be a Petri curve and either $k \geq 5$ or $g\ge5$ and $k\ge3$. Then $\overline B(2,K,k)$ is the closure of $B(2,K,k)$ in $\overline M(2,K)$. In particular, 
$\overline B(2,K,k) = \emptyset$ whenever $B(2,K,k) = \emptyset$. 
\end{prop}

\begin{proof}
It follows from the structure of $B_H(k)$ as a degeneracy locus that every component has dimension $\geq \beta(2,K,k) + 1$. By Lemma \ref{lem2.4} there are
no components of $B_H(k)$ lying entirely in 
$$
\pi_2^{-1}(\overline M(2,K) \setminus M(2,K)).
$$
Note that any point of $\overline B(2,K,k) \setminus B(2,K,k)$ can be represented by a bundle $F = M \oplus (K \otimes M^{-1})$, where $M$ has degree $g-1$ 
and $h^0(M) \geq \lceil \frac{k}{2} \rceil$. If $M \simeq K \otimes M^{-1}$, then $M$ is a theta characteristic, which is impossible on a Petri curve.
Otherwise, the general elementary transformation $0 \ra F \ra E \ra \CC(p) \ra 0$ has $E$ stable and defines a point of $H$. So
$$
\pi_2(B_H(k)) = \overline B(2,K,k)
$$
and the result follows.
\end{proof}

The following proposition is very useful.

\begin{prop} \label{prop2.7}
Let $\cH_g$ denote the Hecke correspondence between $M(2,K)$ and $M(2,K(p))$ for a curve of genus $g$. Suppose that the class $b_H(k)$ is non-zero in $H^*(\cH_{g_0})$ for an integer $g_0 > 0$. 
Then it is non-zero in $H^*(\cH_g)$ for all $g \geq g_0$.
\end{prop}

\begin{proof}
By assumption, either $f(\alpha, \beta, \gamma)$ or $f'(\alpha, \beta, \gamma)$
does not belong to the ideal $I_g$ of relations. We recall from \cite[Lemma 3.1]{kn} that $I_g \subset I_{g_0}$ for all $g \geq g_0$.
Therefore, for all $g \geq g_0$, either 
$f(\alpha, \beta, \gamma)$ or $f'(\alpha, \beta, \gamma)$ does not belong to $I_g$. This implies the assertion.
\end{proof}

\begin{cor} \label{cor2.9}
Suppose that $B(2,K,k) \neq \emptyset$ for a general curve of genus $g_0$. Then $B(2,K,k) \neq \emptyset$ for a general curve of any genus $g \geq g_0$.
\end{cor}

\begin{proof}
 This follows immediately from Theorem \ref{thm2.5}, Remark \ref{rem2.6} and Proposition \ref{prop2.7}.
\end{proof}

\begin{cor}  \label{cor2.10}
Suppose $k \leq 7$ and let $C$ be a general curve of genus $g \geq 3$ such that $\beta(2,K,k) \geq 0$. Then $B(2,K,k) \neq \emptyset$.
\end{cor}

\begin{proof}
By \cite{bf} and \cite{m}, $B(2,K,k) \neq \emptyset$ for the smallest value $g_0$ of the genus for which $\beta(2,K,k) \geq 0$. The result follows from Corollary \ref{cor2.9}. 
Note that for $k = 7$ we have $g_0 = 11$, which is covered by these references.
\end{proof}

\begin{rem} \label{rem2.11}
{\rm For $k = 8$ we need $g_0 \geq 13$. Teixidor's result in \cite{t2} would require $g \geq 16$. The case $k=8, \; g = 13$ seems to be the first unknown case.} 
\end{rem}

\section{Computation of $P_k$}
Let $c(t)=\sum_{n\ge0}c_nt^n$, where the $c_n$ are defined by \eqref{e2.1}$'$-\eqref{e2.3}$'$. Noting that $c_0=2$, we can write formally 
$$c(t)=1 + \prod_{\alpha} (1 + \lambda_{\alpha} t),$$
where the $\lambda_\alpha$ denote the Chern roots corresponding to the Chern character $\ch'$.
\begin{lem} \label{lem3.1}
 $$
 \left( 1 - \frac{\beta}{4}t^2 \right)^2 \frac{d}{dt} (c(t)) = (c(t) - 1) \left( h \left(1 - \frac{\beta}{4}t^2 \right) - \gamma \frac{t^2}{2} \right)
 $$
\end{lem}

\begin{proof}
We have
\begin{eqnarray*}
c(t) - 1 &=& \exp \left( \sum_{\alpha} \log (1 + \lambda_{\alpha} t) \right)\\
&=& \exp \left( \sum_{\alpha} \lambda_{\alpha}t - \sum_{\alpha} \lambda_{\alpha}^2\frac{t^2}{2} + \sum_{\alpha} \lambda_{\alpha}^3\frac{t^3}{3} -+ \cdots \right)\\
&=& \exp \left( \ch'_1 t + \ch'_3 2! t^3 + \cdots + \ch'_{2n+1} (2n)! t^{2n+1} + \cdots \right)\\
&=& \exp \left( \sum_{n \geq 0} \left(\frac{\beta h}{4} - \frac{n}{2} \gamma \right) \left(\frac{\beta}{4} \right)^{n-1} \frac{t^{2n+1}}{2n+1} \right).
\end{eqnarray*}
Differentiating we get
\begin{eqnarray*}
\frac{d}{dt}(c(t) -1) &=& (c(t)-1) \left( \sum_{n \geq 0} \left( \frac{\beta h}{4} - \frac{n}{2} \gamma \right) \left( \frac{\beta}{4} \right)^{n-1} t^{2n} \right)\\
&=& (c(t)-1) \left(h \sum_{n \geq 0} \left( \frac{\beta}{4} \right)^nt^{2n} - \frac{\gamma}{2} t^2 \sum_{n\geq 1} n \left(\frac{\beta}{4} \right)^{n-1} t^{ 2n-2} \right)\\
&=& (c(t) -1) \left( \frac{h}{1-\frac{\beta}{4}t^2} - \frac{\gamma}{2} t^2 \frac{1}{\left(1 - \frac{\beta}{4}t^2 \right)^2 }\right).
\end{eqnarray*}
Multiplying by $\left( 1 - \frac{\beta}{4} t^2 \right)^2$ we get the assertion.
\end{proof}

Inserting $c(t) = 2 + \sum_{n \geq 1} c_nt^n$ and comparing the coefficients of powers of $t$, we get as an immediate consequence the following recurrence relation for the coefficients $c_n$.

\begin{cor} \label{cor3.2}
$$
c_1 = h, \qquad 2c_2 = h^2, \qquad 3c_3 = \frac{1}{2}h^3 + \frac{\beta}{4}h - \frac{\gamma}{2}, \quad 4c_4 = \frac{1}{6}h^4 + \frac{\beta}{3} h^2 - \frac{2\gamma}{3}h
$$
and for all $n \geq 1$,
 $$
 (n+4) c_{n+4} - \frac{\beta}{2}(n+2) c_{n+2} + \left(\frac{\beta}{4} \right)^2 n c_n = h c_{n+3} - \left( \frac{\beta h}{4} + \frac{\gamma}{2} \right) c_{n+1}.
 $$
\end{cor}

\begin{rem} \label{rem3.3}
{\rm Let us write $\widetilde c_i := c_i(1,\beta,0)$. Then the same proof as of Lemma \ref{lem3.1} gives the simpler recurrence relation for $n \geq 1$,}
\begin{equation} \label{remrec}
\widetilde c_0 = 2, \quad \widetilde c_1 = 1, \quad (n+1) \widetilde c_{n+1} = \widetilde c_n + \frac{\beta}{4}(n-1) \widetilde c_{n-1}. 
\end{equation}

\end{rem}

\begin{lem} \label{lemej}
Suppose $\beta = 4$. Then for $n \geq 1$,
$$
\widetilde c_{2n+1} = \widetilde c_{2n}= \frac{(2n)!}{2^{2n} (n!)^2}.
$$
Furthermore, for any odd prime $p > n$,
$$
\widetilde c_{2n+1} \equiv (-1)^n e_n \mod p
$$
where $e_n$ is defined by
$$
(1+t)^{\frac{p-1}{2}} = \sum_{i=0}^{\frac{p-1}{2}}e_i t^i.
$$
\end{lem}

\begin{proof}
The first assertion follows by calculating the low values and proving the general formula by induction using \eqref{remrec}. For the last assertion note that $e_{n} = {\frac{p-1}{2} \choose n}.$
Calculating modulo $p$ we see that 
$$
{\frac{p-1}{2} \choose n} \equiv (-1)^n \frac{(2n)!}{2^{2n} (n!)^2} \mod p.
$$ 
(Note that, since $p>n$, the denominators of both sides of this congruence are coprime to $p$, so the congruence makes sense.) This gives the result.
\end{proof}

\begin{lem} \label{lem3.5}
For any odd prime $p>k$,
$$
P_k(1,4,0) \not \equiv 0 \mod p.
$$
\end{lem}

\begin{proof}
Recall that 
$$
P_k(1,\beta,0) = \left| \begin{array}{ccccc}
                        \widetilde c_k & \widetilde c_{k+1} & \cdots & \cdots & \widetilde c_{2k-1} \\
                       \widetilde c_{k-2} & \widetilde c_{k-1} & \cdots & \cdots& \widetilde c_{2k-3} \\
                        \cdots & \cdots & \cdots & \cdots& \cdots \\
0&\cdots & 0 & 2 & 1
 \end{array}   \right|.
$$
Now suppose $\beta = 4$. Then, by Lemma \ref{lemej}, the last 2 columns are identical except for their
final entries. So we can subtract the last column from the penultimate one and expand by this column 
to get
$$
P_k(1,4,0) = - \left| \begin{array}{cccccc}
 \widetilde c_k & \widetilde c_{k+1} & \cdots & \cdots & \widetilde c_{2k-3} & \widetilde c_{2k-1}\\
\widetilde c_{k-2} & \widetilde c_{k-1} & \cdots & \cdots& \widetilde c_{2k-5} &\widetilde c_{2k-3} \\
            \cdots &  \cdots & \cdots & \cdots & \cdots& \cdots \\
0&\cdots & 0 & 2 & 1& \widetilde c_3
 \end{array}   \right|.
$$
In this matrix the penultimate and antepenultimate columns differ only by their final entries. So we 
can continue in this way and get for $k = 2m+1$,
$$
P_k(1,4,0) = (-1)^{\lfloor \frac{m+1}{2} \rfloor}
                      \left| \begin{array}{cccc}
     \widetilde c_{2m+1} & \widetilde c_{2m+3} & \cdots & \widetilde c_{4m+1} \\
                       \widetilde c_{2m-1} & \widetilde c_{2m+1} & \cdots & \widetilde c_{4m-1} \\
                        \cdots & \cdots & \cdots & \cdots \\
\widetilde c_1&\widetilde c_3 & \cdots & \widetilde c_{2m+1}
 \end{array}   \right|.
$$
and for $k = 2m$
$$
P_k(1,4,0) = (-1)^{\lfloor \frac{m+1}{2} \rfloor}
                      \left| \begin{array}{cccc}
     \widetilde c_{2m+1} & \widetilde c_{2m+3} & \cdots & \widetilde c_{4m-1} \\
                       \widetilde c_{2m-1} & \widetilde c_{2m+1} & \cdots & \widetilde c_{4m-3} \\
                        \cdots & \cdots & \cdots & \cdots \\
\widetilde c_3&\widetilde c_5 & \cdots & \widetilde c_{2m+1}
 \end{array}   \right|.
$$

Now suppose that $p$ is an odd prime, $p > k$. Suppose first $k = 2m+1$.
Then we have, by Lemma \ref{lemej},
\begin{eqnarray*}
P_k(1,4,0) & \equiv &  (-1)^{\lfloor \frac{m+1}{2} \rfloor}
     \left| \begin{array}{cccc}
     e_m & e_{m+1} & \cdots & e_{2m} \\
                       e_{m-1} & e_{m} & \cdots & e_{2m-1} \\
                        \cdots & \cdots & \cdots & \cdots \\
        e_0& e_1& \cdots & e_m
 \end{array}   \right| \\
& \equiv & (-1)^{\lfloor \frac{m+1}{2} \rfloor} \Delta_{m,\dots,m}(e_i) \mod p
\end{eqnarray*}
where $m$ is repeated $m+1$ times. Now
$$
\Delta_{m,\dots,m}(e_i) = S_{m+1,\dots,m+1,0,\dots,0}(1,\dots,1)
$$
where $m+1$ is repeated $m$ times, $0$ is repeated $\frac{p-1}{2} - m$ times and
$S$ is the Schur polynomial (see \cite[equation (A.6)]{fh}). Using \cite[Exercise A.30]{fh}, we see that
$$
S_{m+1,\dots,m+1,0,\dots,0}(1,\dots,1) \not\equiv 0 \mod p.
$$
This completes the proof for $k$ odd. For $k = 2m$ we have similarly
$$
P_k(1,4,0) \equiv (-1)^{\lfloor \frac{m}{2} \rfloor} S_{m,\dots,m,0,\dots,0}(1,\dots,1) \mod p
$$
where $m$ is repeated $m$ times and 0 is repeated $\frac{p-1}{2} -m$ times. Again $S_{m,\dots,m,0,\dots,0}(1,\dots,1) \not\equiv 0\mod p$ which completes the proof.
\end{proof}

\begin{rem} \label{r4.6}
{\rm Using \cite[Exercise A.30]{fh} one can show that 
$$
P_k(1,4,0) \equiv (-1)^{\delta(k)} \frac{1}{2^{\frac{k(k-1)}{2}}} \mod p
$$
for all odd primes $p > k$ and hence 
$$
P_k(1,4,0) = (-1)^{\delta(k)} \frac{1}{2^{\frac{k(k-1)}{2}}}.
$$
Here $\delta(k) = 0$ if $k = 2m+1$ or $k=2m$ and $m$ is even and $\delta(k) = 1$ if $k = 2m$ and $m$ is odd.
}
\end{rem}

\begin{lem} \label{lem3.6}
The polynomial $P_k(1,\beta,0)$ has degree $\leq \lfloor \frac{k^2}{4} \rfloor$. 
\end{lem}

\begin{proof}
The argument is similar to that of Lemma \ref{lem3.5}. From Remark \ref{rem3.3} we see that the maximum power of $\beta$ in $\widetilde c_n$ is at most 
$\lfloor \frac{n-1}{2} \rfloor$. Expanding $P_k(1,\beta,0)$ by its last row we see that the highest powers of $\beta$ arise if we delete the penultimate column and the last row. 
Now continue just as in the proof of Lemma \ref{lem3.5}. We end up with the following patterns of powers of $\beta$:
$$
\left( \begin{array}{cccc}
        m&m+1&\cdots&2m-1\\
        m-1&m&\cdots&2m-2\\
        \cdots&\cdots&\cdots&\cdots\\
        1&2& \cdots&m
       \end{array}  \right)
$$
if $k = 2m$. So the maximum power of $\beta$ in this case is 
$$
1 + 3 + \cdots + (2m-1) = m^2 = \left\lfloor \frac{k^2}{4} \right\rfloor.
$$
If $k = 2m+1$, the pattern is
$$
\left( \begin{array}{cccc}
        m&m+1&\cdots&2m\\
        m-1&m&\cdots&2m-1\\
        \cdots&\cdots&\cdots&\cdots\\
        0&1& \cdots&m
       \end{array}  \right).
$$
In this case the maximum power of $\beta$ is
$$
0 + 2 + \cdots + (2m) = m(m+1) = \frac{k-1}{2} \cdot \frac{k+1}{2} = \left\lfloor \frac{k^2}{4} \right\rfloor.
$$
\end{proof}

\begin{lem} \label{lem3.7}
For $i$ an integer, $1 \leq i < \lfloor \frac{k+1}{2} \rfloor$, the number $\frac{1}{i^2}$ is a zero of the polynomial $P_k(1,\beta,0)$ of multiplicity at least
$\lfloor \frac{k+1}{2} \rfloor -i$.
\end{lem}

\begin{proof}
If we substitute $\beta = \frac{1}{i^2}$, it follows from \eqref{remrec} that for 
$n \geq 1, \; \widetilde c_n$ has the form 
$$
\frac{1}{(2i)^n} \left[ a_0 + a_1n  + \cdots + a_{i-1}n^{i-1} \right].
$$
To see this, substituting in \eqref{remrec} gives $i-1$ homogeneous equations in 
$a_0,\dots,a_{i-1}$ which have a non-trivial solution. This is determined
up to a scalar multiple which is itself determined  by $\widetilde c_1 = 1$.

This implies that the first $\lfloor \frac{k+1}{2} \rfloor$ rows of the matrix defining $P_k(1,\frac{1}{i^2},0)$ lie in a $\QQ$-vector space of dimension $i$.
So the rank of the matrix drops by $\lfloor \frac{k+1}{2} \rfloor -i$. Thus $P_k(1,\beta,0)$ has $\frac{1}{i^2}$ as a root of multiplicity at least 
$\lfloor \frac{k+1}{2} \rfloor -i$.
\end{proof}

This lemma is not best possible. In fact, we make the following conjecture.

\begin{conj}\label{conj}
For $i$ an integer, $1 \leq i \le k-1$, the number $\frac{1}{i^2}$ is a zero of the polynomial $P_k(1,\beta,0)$ of multiplicity at least
$\lfloor \frac{k-i+1}{2} \rfloor $.
\end{conj}
It is easy to see that Conjecture \ref{conj} holds for $i=1$ and we have verified it using Maple for $k\le20$. The conjecture has two interesting consequences.
\begin{itemize}
\item[(i)] The polynomial $P_k(1,\beta,0)$ has degree precisely $\left\lfloor\frac{k^2}4\right\rfloor$ and is non-zero away from $\beta=\frac1{i^2}$.
\item[(ii)] The inequality $g-1 \geq \max \left\{ \frac{k(k-1)}{4}, 2k-1 \right\}$ in Theorem \ref{thm4.3} can be replaced by $g-1 \geq \max \left\{ \frac{k(k-3)}{4}, 2k-1 \right\}$.
\end{itemize}

\section{Some results for $g$ prime}

\begin{lem} \label{lem4.1}
Suppose that $g$ is an odd prime and $m + 2n + 3p = 3g-3$. Then
$$
(\alpha^m \beta^n \gamma^p) \equiv  \left\{ \begin{array}{ll}
                                         -1 \mod g & \mbox{if}\; p=0 \;\mbox{and} \; m= g-1, 2g-2 \; \mbox{or}\; 3g-3,\\
                                         0 \mod g & \mbox{otherwise}.
                                        \end{array} \right.
$$                                        
\end{lem}

\begin{proof}
We use the formulae of Proposition \ref{thaddeus}. 
Recall that $q = m+p+1-g$ which is always even.

Noting that $(g-1)!\equiv-1\mod g$, it follows immediately from von Staudt's Theorem \cite[p. 384]{bs} that
\begin{enumerate}
\item[(a)] $(q+2)(q+1)\left( \frac{q}{2} \right) ! B_q$ is an integer,
\item[(b)] $g (g+1) B_{g-1} \equiv -1 \mod g$ \;  and 
\item[(c)] $2g(2g-1) B_{2g-2} \equiv 2 \mod g$.
\end{enumerate}
 
If $p>0$, then $\frac{g!}{(g-p)!}$ is an integer divisible by $g$. So by (a),
$$
(\alpha^m \beta^n \gamma^p) \equiv 0 \mod g,
$$
except possibly for $q \equiv -1 \mod g$ or $q \equiv -2 \mod g$, i.e. $q = g-1$ or $ q= 2g-2$, since $q$ is even and $0 \leq q \leq 2g-2$. 
If $q=g-1$, we have $m = 2g-2-p$. If $p = g-1$, this is impossible, since then $m=0$. If $p < g-1$, then $g$ divides $m!$ but not $q!$. So from (a) we get 
$(\alpha^m \beta^n \gamma^p) \equiv 0 \mod g$. The case $q = 2g-2$ is impossible for $p >0$.

If $p=0$, we have $m = q + g -1$, so 
$$
(\alpha^m\beta^n) = - \frac{(q + g - 1)!}{q!} 2^{2g-2}(2^q -2) B_q.
$$
If $q = 0$, this gives $(g-1)!2^{2g-2} \equiv -1 \mod g$.
If $2 \leq q \leq g-3$, it follows from (a) that $(\alpha^m\beta^n) \equiv 0 \mod g$.
If $q = g-1$, we have
$$
(\alpha^m \beta^n) = -(2g-2) \cdots (g+2)(g+1)g\ 2^{2g-2}(2^{g-1} -2)B_{g-1}.
$$
By (b),
\begin{eqnarray*}
(\alpha^m \beta^n) &\equiv& (2g-2) \cdots (g+2) 2^{2g-2}(2^{g-1}-2)\\
& \equiv& - \frac{(g-1)!}{g-1} \equiv -1 \mod g.
\end{eqnarray*}
If $g+1 \leq q \leq 2g-4$, we have by (a) $(\alpha^m \beta^n) \equiv 0 \mod g$. 
If $q = 2g-2, \; n=0$ and 
$$
(\alpha^m) = -(3g-3) \cdots (2g+1)2g(2g-1) 2^{2g-2}(2^{2g-2} -2) B_{2g-2}.
$$
So by (c),
\begin{eqnarray*} 
(\alpha^m) &\equiv& -2(3g-3) \cdots (2g+1) 2^{2g-2}(2^{2g-2} -2)\\
& \equiv & 2(g-3)! \equiv - (g-2)! \equiv -1 \mod g.
\end{eqnarray*}
\end{proof}

Recall the polynomial $P_k(h,\beta,\gamma)$ from \eqref{e2.4}. It follows from Corollary \ref{cor3.2} that
$$ 
w := \left((g-1)! 2^{g-1} \right)^k P_k(h, \beta,\gamma) \in \ZZ[h, \beta, \gamma]
$$
for $g > 2k$. We can write
\begin{equation} \label{e4.1}
w = \sum_{j\geq 0} M_j \beta^j h^{\frac{k(k+1)}{2} -2j} + \gamma R(h,\beta,\gamma) 
\end{equation}
with integers $M_j$. Writing 
$$
e := \beta(2,K,k) = 3g-3 - \frac{k(k+1)}{2},
$$
we have using Lemma \ref{hr},
\begin{eqnarray*}
w_0 & := & \alpha h^e w = \sum_{j\geq 0} M_j \alpha \beta^j h^{3g-3-2j} + \alpha h^e \gamma R(h,\beta,\gamma)\\
&=& \sum_{j\geq 0} M_j\alpha \beta^j \frac{1}{2^{3g-3-2j-1}} \left( \sum_{i \; odd} {3g-3-2j \choose i} \alpha^{3g-3-2j-i} \beta^{\frac{i-1}{2}}\right) h\\
&& \hspace{3cm} + \gamma f_2(\alpha,\beta,\gamma)h + g(\alpha, \beta,\gamma)\\
&=& \frac{1}{2^{3g-3}} f_1(\alpha, \beta)h + \gamma f_2(\alpha,\beta,\gamma)h + g(\alpha, \beta,\gamma)\\
\end{eqnarray*}
with 
\begin{eqnarray*}
f_1(\alpha, \beta) &=& \sum_{j\geq 0} \sum_{i \; odd} 2^{2j+1} M_j {3g-3-2j \choose i} \alpha^{3g-2-2j-i} \beta^{j + \frac{i-1}{2}} \in \ZZ[\alpha,\beta].
\end{eqnarray*}

\begin{lem} \label{lem4.2}  Let $g$ be a prime, $g>2k$. Then
$$
\left(f_1(\alpha,\beta) \right) \equiv 2 \left(M_0 + M_{\frac{g-1}{2}} + M_{g-1} \right)  \mod g
$$
\end{lem}

\begin{proof}
By Lemma \ref{lem4.1},
$$
(f_1(\alpha,\beta)) \equiv - \sum_{2j+i-1 = \atop 0,g-1,2g-2}  2^{2j+1}M_j {3g-3-2j \choose i}  \mod g
$$
If $2j+i-1 = 0$, we have  $i=1, j=0$, giving
$$
- 2M_0 (3g-3)  \equiv 6 M_0 \mod g
$$

If $2j+i-1 = g-1$, we have $i = g -2j$. Then
$$
{ 3g-3-2j \choose i} = \frac{(3g-3-2j) \cdots (2g-2)}{i!}.
$$
This has a factor of $g$ unless either $j=0, i=g$ or $j= \frac{g-1}{2},i =1$.
So we get the terms
$$
-\left[ 2M_0 {3g-3 \choose g} + 2^gM_{\frac{g-1}{2}} {2g-2 \choose 1} \right] \equiv  -4M_0 + 4 M_{\frac{g-1}{2}} \mod g
$$
 
If $2j+i-1 = 2g-2$, we have $i = 2g-2j-1$. Then  
$$
{3g-3-2j \choose 2g-2j-1} = \frac{(3g-3-2j) \cdots g(g-1)}{(2g-2j-1)!}.
$$
If $j> \frac{g-1}{2}$, this is divisible by $g$ unless $3g-3-2j = g-1$, i.e. $j=g-1,i=1$.
If $j = \frac{g-1}{2}$, we have ${3g-3-2j \choose 2g-2j-1} = \frac{(2g-2) \cdots (g-1)}{g!}$ which is not divisible by $g$.
If $j< \frac{g-1}{2}$, then  ${3g-3-2j \choose 2g-2j-1}$ is divisible by $g$. So we get the terms
\begin{eqnarray*}
& -& \left[2^{2g-1}M_{g-1}(g-1) + 2^gM_{\frac{g-1}{2}} {2g-2 \choose g} \right]\\
&\equiv& 2M_{g-1} - 2 \frac{(2g-2)\cdots (g+1)(g-1)}{(g-1)!}M_{\frac{g-1}{2}} \\
&\equiv& 2M_{g-1} - 2M_{\frac{g-1}{2}} \mod g.
\end{eqnarray*}
Adding up, this gives the result. 
\end{proof}

Suppose now we write, for $1 \leq \ell \leq \frac{e}{2}$,
$$
w_{\ell} := \alpha \beta^{\ell} h^{e-2\ell} w
$$
and define 
$$
f_{1\ell}(\alpha,\beta) := \sum_{j \geq 0} \sum_{i \;odd} M_j 2^{2j+2\ell+1} {3g-3-2j-2\ell \choose i} \alpha^{3g-2-2j-2\ell-i}\beta^{j+\ell + \frac{i-1}{2}}.
$$
Then the same proof as for $w_0$ gives 
\begin{equation} \label{e4.2}
(f_{1\ell}(\alpha,\beta)) \equiv 2\left( M_{\frac{g-1}{2} - \ell} + M_{g-1-\ell} \right) \mod g. 
\end{equation}

\vspace{1cm}
Define, for $0 \leq i < \frac{g-1}{2}$,
$$
M'_i : \equiv M_i + M_{i+\frac{g-1}{2}} + M_{i+g-1} \mod g
$$
with $0 \leq M'_i \leq g-1$ and consider
$$
q(\beta) := M'_0 + M'_1\beta + \cdots + M'_{\frac{g-3}{2}} \beta^{\frac{g-3}{2}} \in \FF_g[\beta]
$$

Let $x \in \ZZ, \; 1 \leq x \leq g-1$. Using the fact that $x^{g-1} \equiv 1 \mod g$ we see that
$$
P_k(x^2) \equiv q(x^2) \mod g
$$
provided that $M_i = 0$ for $i \geq \frac{3g-3}{2}$.
This is true by Lemma \ref{lem3.6}, since $\frac{3g-3}{2} > \frac{k^2}{4}$. It follows from Lemma \ref{lem3.5} that $q$ is not identically zero.

Moreover, by Lemma \ref{lem3.7}, $q$ has at least $\lfloor \frac{k-1}{2} \rfloor$ distinct zeros different from zero. If $M'_0 = 0$, then $q$ has 
at least $\lfloor \frac{k+1}{2} \rfloor$ zeros. Therefore $\deg q \geq \lfloor \frac{k+1}{2} \rfloor$ which gives $M'_{k_0} \neq 0$ for some 
$k_0 \geq \lfloor \frac{k+1}{2} \rfloor$.

\begin{theorem} \label{thm4.3}
Suppose $g$ is an odd prime. If $g-1 \geq \max \left\{ \frac{k(k-1)}{4}, 2k-1 \right\}$, then $b_H(k) \neq 0$. 
\end{theorem}

\begin{proof}
If $M'_0 \neq 0$, then $b_H(k) \neq 0$ by Lemma \ref{lem4.2}. If $M'_0 = 0$, then $M'_{k_0} \neq 0$ for some $k_0 \geq \lfloor \frac{k+1}{2} \rfloor$.
We have $k_0 < \frac{g-1}{2}$ and we claim that 
$$
\frac{g-1}{2} - k_0 \leq \frac{e}{2}.
$$
In fact, this is equivalent to $g-1 - 2k_0 \leq 3g-3-\frac{k(k+1)}{2}.$
This holds if $g-1-2 \lfloor \frac{k+1}{2} \rfloor \leq 3g-3 - \frac{k(k+1)}{2}$ which is true if $g -1 \geq  \frac{k(k-1)}{4}$. 
The last inequality is true by hypothesis.

So consider $w_{\ell}$ with $\ell = \frac{g-1}{2} - k_0$. Then by \eqref{e4.2}, 
$$
(f_{1\ell}(\alpha,\beta)) \equiv 2\left( M_{k_0} + M_{\frac{g-1}{2} + k_0} \right) \equiv 2M'_{k_0} \mod g
$$
provided that $M_{g-1+k_0} \equiv 0 \mod g$. This is true by Lemma \ref{lem3.6} if $g-1+k_0 > \frac{k^2}{4}$.
In fact, 
$$
g -1+ k_0 \geq g-1+ \left\lfloor \frac{k+1}{2} \right\rfloor \geq \frac{k(k-1)}{4} + \left\lfloor \frac{k+1}{2} \right\rfloor >  \frac{k^2}{4}.
$$
So $b_H(k) \neq 0$. 
\end{proof}

\begin{cor} \label{cor4.4}
If $g$ is a prime $\geq 19$ and $g-1 \geq \frac{k(k-1)}{4}$, then $b_H(k) \neq 0$. Moreover, if $C$ is a Petri curve, then $B(2,K,k) \neq \emptyset$. 
\end{cor}

\begin{proof}
Let $k_1$ be the largest value of $k$ for which $g-1 \geq \frac{k(k-1)}{4}$. Then $k_1 \geq 9$ and therefore $\frac{k_1(k_1-1)}{4} \geq 2k_1 -1$. 
It follows from Theorem \ref{thm4.3} that $b_H(k) \neq 0$ and therefore by Remark \ref{r2.6}
this gives the last assertion.
\end{proof}

\begin{cor} \label{cor4.5}
Given an integer $k \geq 8$, let $g_k$ be the smallest prime satisfying $g_k - 1 \geq \frac{k(k-1)}{4}$. Then $b_H(k) \neq 0$
when $C$ is any curve of genus $g \geq g_k$ ($g$ not necessarily prime). Moreover, if $C$ is a Petri curve, then $B(2,K,k) \neq \emptyset$.
\end{cor}

\begin{proof}
For $k \geq 9$, the first statement is an immediate consequence of the previous corollary and Proposition \ref{prop2.7}. 
For $k = 8$ we have $g_k = 17$ and Theorem \ref{thm4.3} still applies.
The second statement follows from Remark \ref{r2.6}. 
\end{proof}

\section{Some Maple computations}

The condition $g-1 \ge 2k-1$ in Theorem \ref{thm4.3} seems to be essential for the arguments of Section 5. The condition $g-1 \geq \frac{k(k-1)}{4}$ is 
used only in the proof of Theorem \ref{thm4.3}. We can dispense with this condition if we are able to use Lemma \ref{lem4.2} or equation \eqref{e4.2} directly.
The computations appear to be complicated to do by hand, but we can do them using Maple for low values of $k$. 

\begin{theorem} \label{thm6.1}
Let $k$ be an integer, $10 \leq k \leq 24$, and let $g'_k$ be the smallest prime satisfying $3g'_k - 3 \geq \frac{k(k+1)}{2}$. Then $b_H(k) \neq 0$ when $C$ 
is any curve of genus $g \geq g'_k$. Moreover, if $C$ is a Petri curve, then $B(2,K,k) \neq \emptyset$. 
\end{theorem}

\begin{proof}
Note that in all cases we have $g'_k > 2k$. To prove that $b_H(k) \neq 0$, it is sufficient by the arguments of Section 5 to show that
\begin{equation} \label{e6.1}
M_0 + M_{\frac{g'_k - 1}{2}} + M_{g'_k-1} \not \equiv 0 \mod g'_k 
\end{equation}
or that for some $\ell, 1 \leq \ell \leq \frac{e}{2}$,
\begin{equation} \label{e6.2}
M_{\frac{g'_k -1}{2} - \ell} + M_{g'_k -1 - \ell} \not \equiv 0 \mod g'_k 
\end{equation}
(see Lemma \ref{lem4.2} and equation \eqref{e4.2}). This can be done by Maple for specific values of $k$. First we need to calculate 
$$
\widehat c_i := (g'_k - 1)! 2^{g'_k-1} \widetilde c_i \mod g'_k
$$
for $i \leq 2k-1$. Next we write down the matrix 
$$
A := \left( \begin{array}{ccccc}
                        \widehat c_k & \widehat c_{k+1} & \cdots & \cdots & \widehat c_{2k-1} \\
                       \widehat c_{k-2} & \widehat c_{k-1} & \cdots & \cdots& \widehat c_{2k-3} \\
                        \cdots & \cdots & \cdots & \cdots& \cdots \\
0&\cdots & 0 & 2 & 1
 \end{array}   \right)
$$ 
and calculate $\det(A) \mod g'_k$. Using \eqref{e4.1} we compute the integers $M_j \mod g'_k$. We have carried this out using Maple in 
the range $10 \leq k \leq 24$. For $k = 17$ we use \eqref{e6.2} with $\ell = 1$. For the other values \eqref{e6.1} suffices. This proves the first assertion. The second
statement follows from Remark \ref{r2.6}.
\end{proof}

\begin{rem} \label{rem6.2}
{\rm For $k=8$ we have $g'_k = 13$ and the argument fails, since $13 < 2k$. The same applies for $k=9$ when} $g'_k = 17$. 
\end{rem}

\begin{rem} \label{rem6.3}
{\rm In all cases Theorem \ref{thm6.1} improves on the results of Teixidor. In some cases $g'_k$ is the smallest integer for which $3g'_k-3 \geq \frac{k(k+1)}{2}$.
Then the conjecture of Bertram-Feinberg and Mukai is completely proved. This applies when $k = 11,15,16, 20$ and $24$, i.e. $g'_k = 23,41,47,71$ and $101$ respectively}.
\end{rem}

\section{Brill-Noether loci in $M(2,K(p))$}

\begin{theorem} \label{thm7.1}
Let $C$ be a smooth projective complex curve of genus $g \geq 2$ and $p \in C$ a fixed point. For a positive integer $k$, if $b_H(k) \neq 0$, then $B(2,K(p),k) \neq \emptyset$
and every component $X$ has dimension 
$$
\dim(X) \geq \beta(2,K,k) + 1.
$$
If $C$ is general of genus $g \geq 3$, then, for $k \geq 2$,
$$
\dim B(2,K(p),k) \leq  \beta(2,K,k) + k.
$$
 \end{theorem}

\begin{proof}
If $b_H(k) \neq 0$, then $B_H(k) \neq \emptyset$. Moreover, 
$$
\pi_1(B_H(k)) \subset B(2,K(p),k)
$$
where $\pi_1:H \ra M(2,K(p))$ denotes the natural projection. Hence $B(2,K(p),k) \neq \emptyset$ with every component of dimension $\geq \beta(2,K,k)$, since the fibres of $\pi_1$ 
are one-dimensional. To get the extra dimension, we use \cite[Theorem 1.1]{o}, which says that every component $X$ of $B(2,K(p),k)$ has dimension 
\begin{equation}\label{e7.1}
\dim(X)\geq \beta(2,2g-1,k) -g + 
{k-1 \choose 2},
\end{equation}
where 
$$
\beta(2,2g-1,k) = 4g-3-k(k-1)
$$
is the usual Brill-Noether number. A simple calculation shows that the right hand side of \eqref{e7.1} is equal to $\beta(2,K,k) + 1$.

For the last assertion note that 
$$
\pi_1^{-1}(B(2,K(p),k)) \subset B_H(k-1).
$$
By Lemma \ref{lem2.4} and the argument in the second and third paragraph 
of the proof of Theorem \ref{thm2.5} one shows that $B_H(k-1)$ has the expected dimension $\beta(2,K,k-1) +1$. So
$$
\dim B(2,K(p),k) \leq \beta(2,K,k-1) = \beta(2,K,k) + k.
$$
\end{proof}

Note that in the first part of Theorem \ref{thm7.1} there is no restriction on the curve $C$. In particular, if $g$ is an odd prime and $g-1 \geq \max \{ \frac{k(k-1)}{4}, 2k-1 \}$, then $B(2,K(p),k) \neq \emptyset$
by Theorem \ref{thm4.3}. Corollaries \ref{cor4.4} and \ref{cor4.5} also apply as does Theorem \ref{thm6.1}.

\end{document}